\pdfoutput=1
\documentclass{article}
\usepackage[utf8]{inputenc}
\usepackage{amsmath}
\usepackage{amssymb}
\usepackage{amsthm}
\usepackage{tikz-cd}
\usepackage{mathtools}
\usepackage{hyperref}
\usepackage[totalwidth=480pt, totalheight=680pt]{geometry}

\title{On weakly \'etale morphisms}
\author{Aise Johan de Jong and Noah Olander}
\date{}

\newcounter{theorem}
\newtheorem{thm}[theorem]{Theorem}
\newtheorem{corollary}[theorem]{Corollary}
\newtheorem{prop}[theorem]{Proposition}
\newtheorem{lemma}[theorem]{Lemma}
\theoremstyle{definition}

\newtheorem{example}[theorem]{Example}
\newtheorem{remark}[theorem]{Remark}
\newtheorem*{ack}{Acknowledgements}

\begin{document}

\maketitle

\begin{abstract}
    We show that the weakly \'etale morphisms, used to define the pro-\'etale site of a scheme, are characterized by a lifting property similar to the one which characterizes formally \'etale morphisms. In order to prove this, we prove a theorem called Henselian descent which is a ``Henselized version" of the fact that a scheme defines a sheaf for the fpqc topology. Finally, we show that weakly \'etale algebras over regular rings arising in geometry are ind-\'etale and that weakly \'etale algebras do not always lift along surjective ring homomorphisms.
\end{abstract}

\section{Introduction} 

Let $f : X \to Y$ be a morphism of schemes. We say $f$ has the \emph{Henselian lifting property} if for every solid commutative diagram
\begin{equation}
\label{equn-1}
    \begin{tikzcd}
    \mathrm{Spec} (A / I) \ar[r]\ar[d] & X \ar[d, "f"]  \\
    \mathrm{Spec} (A) \ar[r]\ar[ur, dashed, "\exists !"] & Y .
    \end{tikzcd}
\end{equation}
where $(A,I)$ is a Henselian pair, there exists a unique dashed arrow making the diagram commute.

\begin{thm}
\label{thm-weakly-etale-lifts}
A morphism of schemes has the Henselian lifting property if and only if it is weakly \'etale. 
\end{thm}

Recall that a morphism of schemes is weakly \'etale if it is flat and its diagonal is flat. The schemes weakly \'etale over a scheme $S$ form the underlying category of the pro-\'etale site $S_{pro\text{-}\acute{e}t}$ of $S$ as defined in \cite{MR3379634}. Theorem \ref{thm-weakly-etale-lifts} is the natural analogue of the characterization of \'etale morphisms as formally \'etale morphisms which are locally of finite presentation. 

Let $(A, I)$ be a Henselian pair. Let $A \to B$ be a ring map which is weakly \'etale and faithfully flat. Denote $B^h$ and $(B \otimes _A B)^h$ the Henselizations of $B$ and $B \otimes _A B$ with respect to $IB$ and $I(B \otimes _A B)$. 

\begin{thm}
\label{thm-henselian-descent}
Let $X$ be a scheme. With notation as above, the diagram
\begin{equation}
\label{equation-sheaf}
    X(A) \to X(B^h) \rightrightarrows X((B \otimes _A B)^h)
\end{equation}
is an equalizer. 
\end{thm}

Theorem \ref{thm-henselian-descent} arises in the course of the proof of Theorem \ref{thm-weakly-etale-lifts} but we think it is of independent interest. We call it Henselian descent as it resembles the fact that a scheme defines a sheaf for the fpqc topology. Note that Theorem \ref{thm-henselian-descent} does not follow formally from this since one need not have $B^h \otimes _A B^h = (B \otimes _A B)^h$.

\begin{thm}
\label{thm-weakly-etale-over-regular}
Let $A$ be an excellent regular domain which contains a field. Let $A \to B$ be a weakly-\'etale ring map. Then $A \to B$ is ind-\'etale. 
\end{thm}

We use Theorem \ref{thm-weakly-etale-over-regular} to show that there exists a surjective ring map $A \to A/I$ and a weakly \'etale $A/I$-algebra $\overline{B}$ which does not lift to a weakly \'etale $A$-algebra. On the other hand, Bhatt and Scholze proved in \cite[Lemma 2.2.12]{MR3379634} that every ind-\'etale $A/I$-algebra lifts to an ind-\'etale $A$-algebra.

\begin{ack}
We are grateful to Bhargav Bhatt and Shizhang Li for a conversation which led to the results of Section \ref{section-last}.
\end{ack}

\section{Reinterpretation of a result of Gabber} 

In this section we translate some of Gabber's results about \'etale cohomology of Henselian pairs in \cite{gabber-affine-proper} to statements about pro-\'etale cohomology. The default topology under consideration in this section is the pro-\'etale topology and all cohomology groups and pullbacks are with respect to this topology unless otherwise stated. Recall that for a scheme $S$ there is a morphism of sites $\nu : S_{pro\text{-}\acute{e}t} \to S_{\acute{e}t}$ and pro-\'etale sheaves in the essential image of the pullback functor $\nu^{-1}$ are called \emph{classical}, see \cite[Definition 5.1.3]{MR3379634}.

\begin{example}
\label{example-finite-type}
Let $X$ be a scheme locally of finite type over $\mathbf{Z}$. Then the functor $U \mapsto X(U) = \mathrm{Mor}(U, X)$ defines a classical sheaf on $S_{pro\text{-}\acute{e}t}$. More generally, if $X \to S$ is a morphism locally of finite presentation, then the functor $U \mapsto \mathrm{Mor}_{S} (U, X)$ defines a classical sheaf on $S_{pro\text{-}\acute{e}t}$. This follows from the characterization of classical sheaves \cite[Lemma 5.1.2]{MR3379634} and the functorial description of being locally of finite presentation \cite[\href{https://stacks.math.columbia.edu/tag/01ZC}{Tag 01ZC}]{stacks-project}.
\end{example}

\begin{remark}
\label{remark-sections}
If $\mathcal{F}$ is a sheaf on $S_{\acute{e}t}$ and $U \to S$ is an \'etale morphism of schemes then $\Gamma (U, \nu^{-1}\mathcal{F}) = \Gamma(U, \mathcal{F})$ and if $\mathcal{F}$ is an abelian sheaf then $H^p(U, \nu^{-1}\mathcal{F}) = H^p(U, \mathcal{F})$ for every $p$, where in both equations the right hand side is \'etale cohomology and the left is pro-\'etale cohomology. See \cite[Lemma 5.1.2, Corollary 5.1.6]{MR3379634}.
\end{remark}

\begin{lemma}
\label{lemma-gabber}
Let $(A,I)$ be a Henselian pair. Denote $i : \mathrm{Spec}(A/I) \to \mathrm{Spec}(A)$ the inclusion. Let $A \to B$ be a weakly \'etale ring map. Let $\mathcal{F}$ be a classical sheaf on $\mathrm{Spec}(A)_{pro\text{-}\acute{e}t}$. Let $B^h$ denote the Henselization of $B$ with respect to $IB$. Then 
$$
\Gamma (\mathrm{Spec}(B^h), \mathcal{F}) = \Gamma (\mathrm{Spec}(B/IB), i^{-1}\mathcal{F}).
$$
If in addition $\mathcal{F}$ is a torsion abelian sheaf, then
$$
H^p(\mathrm{Spec}(B^h), \mathcal{F}) = H^p(\mathrm{Spec}(B/IB), i^{-1}\mathcal{F})
$$
for every $p$.
\end{lemma}

\begin{remark}
The statement makes sense since $\mathrm{Spec}(B^h)$ is an object of the pro-\'etale site of $\mathrm{Spec}(A)$ whose pullback to $\mathrm{Spec}(A/I)_{pro\text{-}\acute{e}t}$ is $\mathrm{Spec}(B/IB)$. This is because $B \to B^h$ is ind-\'etale and $B/IB = B^h/IB^h$.
\end{remark}

\begin{proof}
If $A = B$ then in view of Remark \ref{remark-sections} this follows from Gabber's results \cite[Theorem 1, Remark 2]{gabber-affine-proper}. A small point is that if $\mathcal{F}$ is an abelian torsion sheaf which is classical then it is $\nu^{-1}$ of an abelian torsion sheaf on $S_{\acute{e}t}$ as follows from Remark \ref{remark-sections}. For general $A \to B$, let $f : \mathrm{Spec}(B^h) \to \mathrm{Spec}(A)$ and $j : \mathrm{Spec}(B/IB) \to \mathrm{Spec}(B^h)$ denote the obvious morphisms. Then $f^{-1}\mathcal{F}$ is a classical sheaf on $\mathrm{Spec}(B^h)_{pro\text{-}\acute{e}t}$ by \cite[Lemma 5.4.1]{MR3379634} so we know by the first sentence of the proof that 
$$
\Gamma (\mathrm{Spec}(B^h), f^{-1}\mathcal{F}) = \Gamma(\mathrm{Spec}(B/IB), j^{-1}f^{-1}\mathcal{F}).
$$
Since $\mathrm{Spec}(B^h)$ is an object of the pro-\'etale site of $\mathrm{Spec}(A)$ the left hand side is just $\Gamma (\mathrm{Spec}(B^h), \mathcal{F})$ by \cite[Lemma 4.2.7]{MR3379634}. Similarly, if $\bar{f}$ denotes the morphism $\mathrm{Spec}(B/IB) \to \mathrm{Spec}(A/I)$ then the right hand side is
$$
\Gamma (\mathrm{Spec}(B/IB), \bar{f}^{-1} i^{-1} \mathcal{F}) = \Gamma (\mathrm{Spec}(B/IB), i^{-1}\mathcal{F}),
$$
since $\mathrm{Spec}(B/IB)$ is an object of $\mathrm{Spec}(A/I)_{pro\text{-}\acute{e}t}$. If $\mathcal{F}$ is an abelian torsion sheaf, the same argument verbatim with $\Gamma$ replaced by $H^p$ proves the second statement. 
\end{proof}

\section{Henselian descent}

In this section, $(A, I)$ always denotes a Henselian pair. For a weakly \'etale $A$-algebra $B$, we let $B^h$ denote the Henselization of $B$ with respect to $IB$. For such a $B$ we obtain a diagram
\begin{equation}
\label{equation-equalizer}
 A \to B^h \rightrightarrows (B \otimes _A B)^h .    
\end{equation}
Our goal is to prove Theorem \ref{thm-henselian-descent} which implies in particular (by taking $X = \mathbf{A}^1_{\mathbf{Z}}$) that if $A \to B$ is weakly \'etale and faithfully flat, then (\ref{equation-equalizer}) is an equalizer.

\begin{lemma}
\label{lemma-faithfully-flat}
If $A \to B$ is weakly \'etale and faithfully flat, then so is $A \to B^h$. If $f: X \to \mathrm{Spec}(A)$ is a flat morphism of schemes whose image contains $V(I)$, then $f$ is faithfully flat. 
\end{lemma}

\begin{proof}
Since $A \to B \to B^h$ are all weakly \'etale (in particular, flat) and $A/I \to B^h/IB^h = B/IB$ is faithfully flat the first statement will follow if we prove the second. The second statement is true because $f(X)$ is closed under generalization and contains $V(I)$, and $I$ is contained in the Jacobson radical of $A$ so every point of $\mathrm{Spec}(A)$ is the generalization of a point of $V(I)$.
\end{proof}


\begin{proof}[Proof of Theorem \ref{thm-henselian-descent}]
By Lemma \ref{lemma-faithfully-flat} and faithfully flat descent, it follows that $X(A) \to X(B^h)$ is injective, so the problem is to show that an element $g \in X(B^h)$ equalized by the two arrows to $X((B \otimes _A B)^h)$ comes from an element of $X(A)$. 

\underline{Step 1.} Reduction to the case $X$ is quasi-compact and quasi-separated. 

We immediately reduce to the case $X$ is quasi-compact by replacing $X$ with a quasi-compact open containing the image of $g : \mathrm{Spec}(B^h) \to X$. A quasi-compact scheme $X$ when viewed as an fppf sheaf can be written as a filtered colimit $\mathrm{colim}_{i} X_i$ with $X_i$ a quasi-compact and quasi-separated scheme and each $X_i \to X$ a local isomorphism. In particular, for every quasi-compact scheme $T$ one has $X(T) = \mathrm{colim}_i X_i(T)$. See \cite[\href{https://stacks.math.columbia.edu/tag/03K0}{Tag 03K0}]{stacks-project}. Then if we know (\ref{equation-sheaf}) is exact when $X$ is replaced by $X_i$, taking the colimit over $i$ proves the result since filtered colimits are exact in the category of sets. 

\underline{Step 2.} Reduction to the case $X$ is of finite type over $\mathbf{Z}$. 

Now assume $X$ is a quasi-compact and quasi-separated scheme. Then by absolute Noetherian approximation \cite[\href{https://stacks.math.columbia.edu/tag/01ZA}{Tag 01ZA}]{stacks-project}, $X$ is a cofiltered limit with affine transition maps $\mathrm{lim}_{i \in I} X_i$ where each $X_i$ is a scheme of finite type over $\mathbf{Z}$. Then if we know (\ref{equation-sheaf}) is exact when $X$ is replaced by $X_i$, taking the limit over $i$ proves the result for $X$. 

\underline{Step 3.} Proof when $X$ is of finite type over $\mathbf{Z}$.

By Example \ref{example-finite-type}, $X$ represents a classical sheaf $\mathcal{F}$ on the pro-\'etale site of $\mathrm{Spec}(A)$. Let $i : \mathrm{Spec}(A/I) \to \mathrm{Spec}(A)$ denote the inclusion. Then by Lemma \ref{lemma-gabber}, the diagram (\ref{equation-sheaf}) is identified with
$$
\Gamma (\mathrm{Spec}(A/I), i^{-1}\mathcal{F}) \to \Gamma (\mathrm{Spec}(B/IB), i^{-1}\mathcal{F}) \rightrightarrows \Gamma (\mathrm{Spec}(B/IB \otimes _{A/I} B / IB), i^{-1}\mathcal{F})
$$
which is an equalizer since $i^{-1}\mathcal{F}$ is a sheaf on the pro-\'etale site of $\mathrm{Spec}(A/I)$.
\end{proof}

In analogy with classical descent theory, it is natural to ask whether the complex
$$
\check{\mathcal{C}}(B/A)^h =  B^h \to (B \otimes _A B)^h \to (B \otimes _A B \otimes _A B)^h \to  \cdots 
$$
is exact in positive degrees. This is the complex obtained from the Amitsur complex $\check{\mathcal{C}}(B/A)$ by Henselizing the terms $B \otimes _A \cdots \otimes _A B$ with respect to $I(B \otimes _A \cdots \otimes _A B)$. In fact, we will show now that this exactness can fail even for standard Zariski coverings $B = A_{f_1} \times \cdots \times A_{f_n}$ where $f_i \in A$ generate the unit ideal.

\begin{example}
\label{example-higher-cohomology}
According to the blog post \cite{dejong_2018} there is a sheaf of rings $\mathcal{O}^h$ on the Zariski site of $\mathrm{Spec}(A/I)$ characterized by $\Gamma (D(f)\cap \mathrm{Spec}(A/I), \mathcal{O}^h) = A_f^h$ for $f \in A,$ where $A_f ^h$ is the Henselization of $A_f$ with respect to $I A_f$. In loc. cit., an example is constructed of a Henselian pair $(A, I)$ with $H^1_{Zar}(\mathrm{Spec}(A/I), \mathcal{O}^h) \neq 0$. Since \v{C}ech cohomology and cohomology always agree in degree 1, we have
$$
0 \neq H^1_{Zar} (\mathrm{Spec}(A/I), \mathcal{O}^h)) = \mathrm{colim}_{A \to B} H^1 (\check{\mathcal{C}}(B/A)^h)
$$
where the colimit is taken over standard Zariski coverings $A \to B$. Since filtered colimits are exact, this implies that there exists a standard Zariski covering $A \to B$ for which the complex $\check{\mathcal{C}}(B/A)^h$ is not exact in degree 1. 
\end{example}

In positive characteristic, such an example is not possible essentially by Gabber's affine analog of proper base change.

\begin{lemma}
Assume $A$ is a $\mathbf{Z}/N$-algebra for some integer $N>1$. Let $M$ be an $A$-module. Then the complex of $A$-modules
\begin{equation}
\label{equn-h-amitsur}
M \otimes _A \check{\mathcal{C}}(B/A)^h = M \otimes _A B^h \to M \otimes _A (B \otimes _A B)^h \to M \otimes _A (B \otimes _A B \otimes _A B)^h \to \cdots 
\end{equation}
is quasi-isomorphic to $M$ placed in degree zero. 
\end{lemma}

\begin{proof}
Let $\mathcal{F}$ be the sheaf on $\mathrm{Spec}(A)_{pro\text{-}\acute{e}t}$ characterized by the formula $\mathcal{F}(U) = M \otimes _A B$ for $U = \mathrm{Spec}(B)$ an affine object of $\mathrm{Spec}(A)_{pro\text{-}\acute{e}t}$. It follows from \cite[\href{https://stacks.math.columbia.edu/tag/023M}{Tag 023M}]{stacks-project} that $\mathcal{F}$ is indeed a sheaf and that for $U$ an affine object of $\mathrm{Spec}(A)_{pro\text{-}\acute{e}t}$ we have $H^p(U, \mathcal{F}) = 0$ for $p>0$. See \cite[\href{https://stacks.math.columbia.edu/tag/03F9}{Tag 03F9}]{stacks-project} for the vanishing. One sees using the description \cite[Lemma 5.1.2]{MR3379634} of classical sheaves that  $\mathcal{F}$ is classical. Moreover, $\mathcal{F}$ is torsion since $N = 0$ in $A$.  Therefore Gabber's result, Lemma \ref{lemma-gabber}, implies that for $A \to B$ weakly \'etale,
\begin{equation}
\label{equn-vanishing}
H^p(\mathrm{Spec}(B/IB), i^{-1}\mathcal{F}) = 
\begin{cases}
M \otimes _A B^h &\text{if } p = 0\\
0 &\text{otherwise.}
\end{cases}
\end{equation}
Since this vanishing also holds when $B$ is replaced by any tensor product $B \otimes _A \cdots \otimes _A B$, a classical result on \v{C}ech cohomology (\cite[\href{https://stacks.math.columbia.edu/tag/03F7}{Tag 03F7}]{stacks-project}) shows that the \v{C}ech complex
$\check{\mathcal{C}}(\{\mathrm{Spec}(B/IB) \to \mathrm{Spec}(A/I)\}, i^{-1}\mathcal{F})$
is quasi-isomorphic to $R\Gamma (\mathrm{Spec}(A/I), i^{-1}\mathcal{F}),$ which is  $M[0]$ by (\ref{equn-vanishing}) applied to $A = B$. On the other hand, the $p=0$ case of (\ref{equn-vanishing}) applied with $B$ replaced by $B \otimes_A \cdots \otimes _A B$ shows that this \v{C}ech complex coincides with the complex (\ref{equn-h-amitsur}).
\end{proof}

\section{Lifting for weakly \'etale morphisms}

We now turn to the proof of Theorem \ref{thm-weakly-etale-lifts}, starting with some formal consequences of the Henselian lifting property.


\begin{lemma}
\label{lemma-formal}
\begin{enumerate}
    \item An open immersion has the Henselian lifting property.
    \item A composition of morphisms having the Henselian lifting property has the Henselian lifting property.
    \item If $f$ and $g$ are composable morphisms of schemes such that $f$ and $f \circ g$ have the Henselian lifting property, then so has $g$.
    \item The base change of a morphism with the Henselian lifting property has the Henselian lifting property.
\end{enumerate}
\end{lemma}

\begin{proof}
We prove 1; 2-4 are formal. It suffices to show that if we are given a diagram (\ref{equn-1}) with $f$ an open immersion, then the image of $\mathrm{Spec}(A) \to Y$ is contained in the open $f(X) \subset Y$. This is because the closed $V(I)$ maps into $f(X)$ and every point of $\mathrm{Spec}(A)$ is a generalization of a point of $V(I)$ since $I$ is contained in the Jacobson radical of $A$.
\end{proof}

\begin{proof}[Proof of $\implies$ direction of Theorem \ref{thm-weakly-etale-lifts}]
If $f$ has the Henselian lifting property then so does its diagonal by Lemma \ref{lemma-formal}. Hence it suffices to show $f$ is flat. If $U \subset X$ is an affine open mapping into an affine open $V \subset Y$ then it follows from Lemma \ref{lemma-formal} that $U \to V$ has the Henselian lifting property. Thus it suffices to show that a ring map $R \to S$ with the Henselian lifting property is flat. Choose a surjection $P \to S$ of $R$-algebras with $P$ a possibly infinite polynomial algebra over $R$. Let $P^h$ denote the Henselization of $P$ with respect to $\mathrm{Ker}(P \to S)$. Then form the solid diagram
\begin{equation*}
\begin{tikzcd}
  S   & S \ar[l, "="'] \ar[dl, dashed] \\
    P^h \ar[u] & R \ar[u] \ar[l]
\end{tikzcd}
\end{equation*}
and fill in the dashed arrow. Now $S$ is a direct summand of $P^h$ as an $R$-module, and $P^h$ is flat over $R$, completing the proof.
\end{proof}

Flat immersions of schemes which are not open immersions are somewhat rare, but they do occur: The diagonal of a weakly \'etale morphism is often an example. We point out here that they share the following property in common with the open immersions:

\begin{lemma}
\label{lemma-flat-immersion}
Let $j : Z \to X$ be a flat immersion of schemes. Let $f : T \to X$ be a morphism. Then $f$ factors through $Z$ (necessarily uniquely) if and only if $f(T) \subset j(Z)$ set-theoretically.
\end{lemma}

\begin{proof}
The only if direction is clear. The if direction is local so we may assume given a flat surjection of rings $A \to A/I$ and a ring map $A \to B$ such that $IB \subset \mathfrak{p}$ for every $\mathfrak{p} \in \mathrm{Spec}(A)$. We have to show $A \to B$ factors through $A/I$. By  \cite[\href{https://stacks.math.columbia.edu/tag/04PS}{Tag 04PS}]{stacks-project}, a flat surjection $A \to A/I$ of rings satisfies $A/I \cong (1 + I)^{-1} A$ as $A$-algebras. By the universal property of localization, it suffices now to show every element $1 + f$ with $f \in I$ is mapped to a unit of $B$. This is because $1 + f$ is not contained in any prime of $B$ by the assumption that they all contain $IB \ni f$. 
\end{proof}

\begin{corollary}
\label{corollary-unique}
Let $f : X \to Y$ be a morphism of schemes.
\begin{enumerate}
    \item If $f$ is a flat immersion then $f$ has the Henselian lifting property.
    \item If the diagonal of $f$ is flat and $(A,I)$ is a henselian pair, then any dashed arrow in a diagram (\ref{equn-1}), if it exists, is unique.
\end{enumerate}
\end{corollary}

\begin{proof}
1. Suppose given a solid diagram (\ref{equn-1}) with $(A,I)$ a Henselian pair. We have to show there exists a unique dashed arrow. By Lemma \ref{lemma-flat-immersion} it suffices to show $\mathrm{Spec}(A)$ maps into the subset $f(X) \subset Y$. But $f$ is flat so $f(X)$ is closed under generalization and contains the image of $V(I) \subset \mathrm{Spec}(A)$, so we conclude as in the proof of Lemma \ref{lemma-formal}.

2. If the diagonal of $f$ is flat then it is a flat immersion so by 1 it has the Henselian lifting property. By a formal argument, this implies uniqueness of dashed arrows for all diagrams of the form (\ref{equn-1}).
\end{proof}

The idea for proving the more difficult direction of Theorem \ref{thm-weakly-etale-lifts} is to use the uniqueness just proven together with Henselian descent to reduce to constructing a dashed arrow locally. 

\begin{proof}[Proof of $\impliedby$ direction of Theorem \ref{thm-weakly-etale-lifts}]
It remains to prove that if $X \to Y$ is weakly \'etale, then it has the Henselian lifting property. To prove this we have:

\underline{Claim.} It is enough to show that if $(A,I)$ is a Henselian pair and $f: X \to \mathrm{Spec}(A)$ is a weakly \'etale, faithfully flat morphism with $X$ quasi-compact and quasi-separated, then for any solid diagram
\begin{equation}
\label{equn-2}
\begin{tikzcd}
    \mathrm{Spec} (A / I) \ar[r, "t"]\ar[d] & X \ar[d, "f"]  \\
    \mathrm{Spec} (A) \ar[r, "="]\ar[ur, dashed, "\exists "] & \mathrm{Spec}(A)
    \end{tikzcd}
\end{equation}
there exists a dashed arrow.

\emph{Proof of Claim.} Assume given a diagram (\ref{equn-1}) with $X \to Y$ weakly \'etale. We only have to show a dashed arrow exists since uniqueness follows from Corollary \ref{corollary-unique}.2. Producing a dashed arrow in (\ref{equn-1}) is equivalent to producing a dashed arrow in the diagram (\ref{equn-2}) with $X$ replaced by $X \times _Y \mathrm{Spec}(A)$. Thus we see it suffices to find a dashed arrow in the diagram (\ref{equn-2}) for $f: X \to \mathrm{Spec}(A)$ weakly \'etale. To do this, we may assume $X$ is quasi-compact by replacing $X$ by a quasi-compact open containing the image of $\mathrm{Spec}(A/I)$. Then write $X = \mathrm{colim}_i X_i$ as in Step 1 of the proof of Theorem \ref{thm-henselian-descent} and replace $X$ by any $X_i$ through which $\mathrm{Spec}(A/I) \to X$ factors to assume $X$ is quasi-compact and quasi-separated, and still weakly \'etale over $A$ since $X_i \to X$ is a local isomorphism. It is automatic from Lemma \ref{lemma-faithfully-flat} that $f$ is faithfully flat. This proves the claim.

Now suppose given a diagram (\ref{equn-2}) as in the claim. Choose a surjective \'etale morphism $g: \mathrm{Spec}(B) \to X$. Then $A \to B$ is faithfully flat and weakly \'etale. It may not be true that $g$ and $t : \mathrm{Spec}(A/I) \to X$ give rise to the same morphism $\mathrm{Spec}(B/IB) \to X$, and we will now remedy this. The cartesian diagram
$$
\begin{tikzcd}[column sep=tiny]
\mathrm{Spec}(B) \times _X \mathrm{Spec}(A/I) \ar[rr] \ar[d] & & \mathrm{Spec}(B) \times _A \mathrm{Spec}(A/I) \ar[d] \ar[r, equal] & \mathrm{Spec}(B/IB)  \\
X \ar[rr, "\Delta_f"] & & X \times _A X
\end{tikzcd}
$$
shows that the top horizontal arrow is weakly \'etale and quasi-compact since $\Delta_f$ is by assumption. Choose surjective \'etale morphism $\mathrm{Spec}(D) \to \mathrm{Spec}(B) \times _X \mathrm{Spec}(A/I)$. Then $B/IB \to D$ is weakly \'etale and there is a commutative diagram
$$
\begin{tikzcd}
\mathrm{Spec}(D) \ar[r] \ar[d] &\mathrm{Spec}(B) \ar[d, "g"] \\
\mathrm{Spec}(A/I) \ar[r, "t"] & X
\end{tikzcd}.
$$
In fact, by \cite[Theorem 2.3.4]{MR3379634}, after replacing $D$ with a faithfully flat, ind-\'etale $D$-algebra, we may even assume $B/IB \to D$ is ind-\'etale. In this case, by \cite[Lemma 2.2.12]{MR3379634}, there is an ind-\'etale $B$-algebra $B'$ with an isomorphism $B'/IB' = D$. Replace $B$ with $B'$. Then $g: \mathrm{Spec}(B) \to X$ and $t : \mathrm{Spec}(A/I) \to X$ indeed give rise to the same morphism $\mathrm{Spec}(B/IB) \to X$ and furthermore, $A \to B$ is weakly \'etale by construction and faithfully flat by Lemma \ref{lemma-faithfully-flat}.

We are looking for an element of $\mathrm{Mor}_A(\mathrm{Spec}(A), X)$ mapping to the given element $t \in \mathrm{Mor}_A(\mathrm{Spec}(A/I), X)$. Set $C = B \otimes _A B$ and let $B^h, C^h$ denote the henselizations of $B, C$ with respect to $IB, IC$. Consider the diagram
$$
\begin{tikzcd}
\mathrm{Mor}_A(\mathrm{Spec}(A), X) \ar[r] \ar[d] &\mathrm{Mor}_A(\mathrm{Spec}(B^h), X) \ar[r, shift left] \ar[r, shift right] \ar[d] &\mathrm{Mor}_A(\mathrm{Spec}(C^h), X) \ar[d]
 \\
\mathrm{Mor}_A(\mathrm{Spec}(A/I), X) \ar[r] &\mathrm{Mor}_A(\mathrm{Spec}(B/IB), X) \ar[r, shift left] \ar[r, shift right] &\mathrm{Mor}_A(\mathrm{Spec}(C/IC), X).
\end{tikzcd}
$$
It follows from faithfully flat descent that the bottom row is an equalizer and from Henselian descent, Theorem \ref{thm-henselian-descent}, that the top row is also. It follows from Corollary \ref{corollary-unique}.2 that the vertical arrows are injective. The morphism $g : \mathrm{Spec}(B) \to X$ induces an element $g' \in \mathrm{Mor}_A(\mathrm{Spec}(B^h), X)$ which maps to the same element of $\mathrm{Mor}_A(\mathrm{Spec}(B/IB), X)$ as $t \in  \mathrm{Mor}_A(\mathrm{Spec}(A/I), X)$. It is now a diagram chase to produce the desired element of $\mathrm{Mor}_A( \mathrm{Spec}(A), X)$.
\end{proof}



\section{Weakly \'etale over regular}
\label{section-last}

\begin{prop}
\label{prop-pro-finite-etale}
Let $Y$ be a Noetherian integral scheme which is regular and Nagata. Let $f: X \to Y$ be an integral, weakly \'etale morphism of schemes. Then $f$ is a co-filtered limit of finite-\'etale morphisms. 
\end{prop}

\begin{proof}
Let $K$ denote the function field of $Y$. Then the generic fibre $X_K = \mathrm{Spec}(A)$ where $A$ is a weakly \'etale $K$-algebra. Thus $A$ is the filtered colimit of its finite-type $K$-subalgebras $A'$ which are all \'etale algebras over $K$ by \cite[\href{https://stacks.math.columbia.edu/tag/0CKR}{Tag 0CKR}]{stacks-project}. On the other hand, $f$ is the integral closure of $Y$ in the $K$-algebra $A$ since $f$ is integral and $X$ is integrally closed in $A$ by \cite[\href{https://stacks.math.columbia.edu/tag/092W}{Tag 092W}]{stacks-project}. Hence $f$ is the co-filtered limit with affine transition maps of the integral closures $X' \to Y$ of $Y$ in $A'$. It suffices to show each $X' \to Y$ is finite \'etale. Finiteness follows from the assumption that $Y$ is Nagata, see \cite[\href{https://stacks.math.columbia.edu/tag/03GH}{Tag 03GH}]{stacks-project}.  For \'etaleness, since $X'$ is normal we may use purity of the branch locus \cite[\href{https://stacks.math.columbia.edu/tag/0BMB}{Tag 0BMB}]{stacks-project}. It suffices to show that for each codimension one point $x' \in X'$ with image $y \in Y$ the extension $\mathcal{O}_{Y,y} \to \mathcal{O}_{X',x'}$ of DVRs is unramified with separable residue field extension. Now $X \to X'$ is surjective: Its image is dense since $A' \subset A$ and closed since $f$ was assumed integral. Thus we may find $x \in X$ mapping to $x'$ and the composition $\mathcal{O}_{Y,y} \to \mathcal{O}_{X',x'} \to \mathcal{O}_{X,x}$ is a weakly \'etale ring map by \cite[\href{https://stacks.math.columbia.edu/tag/094Q}{Tag 094Q}]{stacks-project}. Thus $\mathcal{O}_{Y,y} \to \mathcal{O}_{X,x}$ induces an isomorphism on strict Henselizations by Olivier's Theorem \cite[\href{https://stacks.math.columbia.edu/tag/092Z}{Tag 092Z}]{stacks-project}, so if
$\pi \in \mathcal{O}_{Y,y}$ is a uniformizer, then the image of $\pi$ in $\mathcal{O}_{X,x}$ must be in $\mathfrak{m}_x \setminus \mathfrak{m}_x ^2$. But then the same must be true for the image of $\pi$ in $\mathcal{O}_{X',x'}$ so that $\mathcal{O}_{Y,y} \to \mathcal{O}_{X',x'}$ is unramified. For the residue field extension, we have a tower $\kappa (x) / \kappa (x') / \kappa (y)$ and $\kappa (x)/\kappa(y)$ is separable algebraic by \cite[\href{https://stacks.math.columbia.edu/tag/092R}{Tag 092R}]{stacks-project} hence so is $\kappa (x')/\kappa (y).$
\end{proof}

By slightly modifying this argument we will prove Theorem \ref{thm-weakly-etale-over-regular}. In place of classical Zariski--Nagata purity \cite[\href{https://stacks.math.columbia.edu/tag/0BMB}{Tag 0BMB}]{stacks-project} we will use the following refinement.

\begin{thm}[{\cite[\href{https://stacks.math.columbia.edu/tag/0ECD}{Tag 0ECD}]{stacks-project}}]
\label{thm-afineness}
Let $Y$ be an excellent regular scheme over a field. Let $f: X \to Y$ be a finite type morphism of schemes with $X$ normal. Let $V \subset X$ be the maximal open subscheme where $f$ is \'etale. Then the inclusion morphism $V \to X$ is affine.
\end{thm}

Next, we will improve the ad-hoc argument from the proof of Proposition \ref{prop-pro-finite-etale} that $\mathcal{O}_{Y,y} \to \mathcal{O}_{X',x'}$ is unramified with separable residue field extension.

\begin{lemma}
\label{lemma-unramified}
Let $A \to B \to C$ be local homomorphisms of local rings. Assume $A \to B$ is essentially of finite presentation, $B$ is a geometrically unibranch domain (for example if $B$ is normal), $A \to C$ is weakly \'etale, and the generic point of $\mathrm{Spec}(B)$ lies in the image of $\mathrm{Spec}(C) \to \mathrm{Spec}(B)$. Then $A \to B$ is the localization of an \'etale ring map. 
\end{lemma}

\begin{proof}
This result is known if instead of assuming $A \to C$ is weakly \'etale we assume it is the localization of an \'etale ring map, see \cite[\href{https://stacks.math.columbia.edu/tag/0GSA}{Tag 0GSA}]{stacks-project}. We will reduce to this case. First note that the hypotheses do not change if we replace $C$ with its strict Henselization $C^{sh}$. This is because $C \to C^{sh}$ is ind-\'etale so that the composition $A \to C^{sh}$ remains weakly \'etale, and $C \to C^{sh}$ is faithfully flat so that the generic point of $\mathrm{Spec}(B)$ is still hit. Combining this with Olivier's Theorem \cite[\href{https://stacks.math.columbia.edu/tag/092Z}{Tag 092Z}]{stacks-project} we see that we may assume $C = A^{sh}$ is the strict henselization of $A$. Then as $A$-algebras, $C = \mathrm{colim}_i C_i$ is a filtered colimit of local $A$-algebras which are localizations of \'etale $A$-algebras, and $B$ is essentially of finite presentation, hence the map $B \to C$ factors through some $C_i$. Replace $C$ with $C_i$ and we are done.
\end{proof}


\begin{proof}[Proof of Theorem \ref{thm-weakly-etale-over-regular}]
As $B$ is the filtered colimit of its finite type $A$-subalgebras $C$, it suffices to show each inclusion $C \to B$ factors through an \'etale $A$-algebra. Let $K$ denote the fraction field of $A$. We have a commutative diagram
$$
\begin{tikzcd}
B \ar[r, hook] &B \otimes _A K  \\
C \ar[r] \ar[u, hook] & C \otimes _A K . \ar[u, hook]
\end{tikzcd}
$$
The top horizontal arrow is injective since $B$ is flat over $A$ and the right vertical arrow is injective since $K$ is flat over $A$. We conclude from the diagram that $C \to C \otimes _A K$ is also injective. Let $C'$ denote the integral closure of $C$ inside $C \otimes _A K$. Then $C' \subset B$ since $B$ is integrally closed in $B \otimes _A K$ by \cite[\href{https://stacks.math.columbia.edu/tag/092W}{Tag 092W}]{stacks-project}. Also $C \to C'$ is finite since $A$ is excellent and $C \otimes _A K$ is reduced, see \cite[\href{https://stacks.math.columbia.edu/tag/0CKR}{Tag 0CKR}]{stacks-project} and \cite[\href{https://stacks.math.columbia.edu/tag/03GH}{Tag 03GH}]{stacks-project}. Replace $C$ by $C'$ to assume $C$ is a normal ring, still of finite type over $A$. Let $V \subset \mathrm{Spec}(C)$ denote the \'etale locus of the ring map $A \to C$. Then $V = \mathrm{Spec}(D)$ is affine by Theorem \ref{thm-afineness} and we will be done if we show the inclusion $C \to B$ factors through $D$. This is equivalent to showing that the image of $\mathrm{Spec}(B) \to \mathrm{Spec}(C)$ is contained in $V$. If $\mathfrak{q} \in \mathrm{Spec}(C)$ is in the image then there is $\mathfrak{p} \in \mathrm{Spec}(B)$ mapping to $\mathfrak{q}$ and then applying Lemma \ref{lemma-unramified} to the local ring homomorphisms $A_{\mathfrak{q} \cap A} \to C_{\mathfrak{q}} \to B_{\mathfrak{p}}$ shows that $A \to C$ is \'etale at $\mathfrak{q}$ as needed.
\end{proof}

\begin{corollary}
\label{corollary-no-lift}
There exists a surjective ring homomorphism $A \to A/I$ and a weakly \'etale $A/I$-algebra which does not lift to a weakly \'etale $A$-algebra.
\end{corollary}

\begin{proof}
We take $A = \mathbf{C}[x,y]$ and $I = ((y-x^2+1)(y+x^2-1)).$ In \cite[\href{https://stacks.math.columbia.edu/tag/09AP}{Tag 09AP}]{stacks-project} we find an example of a weakly \'etale $A/I$-algebra $\overline{B}$ which is not ind-\'etale. If there were a lift to a weakly \'etale $A$-algebra $B$ then $A \to B$ would necessarily be ind-\'etale by Theorem \ref{thm-weakly-etale-over-regular}. But then $A/I \to B/IB = \overline{B}$ would be ind-\'etale also, a contradiction.
\end{proof}

\bibliographystyle{alpha}
\bibliography{references}

\begin{thebibliography}{{Sta}22}

\bibitem[BS15]{MR3379634}
Bhargav Bhatt and Peter Scholze.
\newblock The pro-\'{e}tale topology for schemes.
\newblock {\em Ast\'{e}risque}, (369):99--201, 2015.

\bibitem[dJ18]{dejong_2018}
Aise~J. de~Jong.
\newblock No \text{T}heorem \text{B} for henselian affine schemes. \emph{Stacks
  Project Blog}.
\newblock \url{https://www.math.columbia.edu/~dejong/wordpress/?p=4262}, Nov
  2018.

\bibitem[Gab94]{gabber-affine-proper}
Ofer Gabber.
\newblock Affine analog of the proper base change theorem.
\newblock {\em Israel J. Math.}, 87(1-3):325--335, 1994.

\bibitem[{Sta}22]{stacks-project}
The {Stacks Project Authors}.
\newblock \textit{Stacks Project}.
\newblock \url{https://stacks.math.columbia.edu}, 2022.

\end{thebibliography}

\textsc{Columbia University Department of Mathematics, 2990 Broadway, New York, NY 10027}

Email: \href{mailto:dejong@math.columbia.edu}{dejong@math.columbia.edu}

\vspace{.5em}

\textsc{Columbia University Department of Mathematics, 2990 Broadway, New York, NY 10027}

Email: \href{mailto:nolander@math.columbia.edu}{nolander@math.columbia.edu}

\end{document}